\documentclass[12pt]{amsart}
			\usepackage{fixmath}
			
			\usepackage[T1]{fontenc}
			\usepackage{ucs}
			\usepackage[utf8x]{inputenc}
			\usepackage[english]{babel}
			
			\usepackage{microtype}
			\usepackage{babel}

			\usepackage{mathtools} 

			\usepackage{enumerate}

			\usepackage{geometry} 

			\usepackage{amsfonts, amsmath, amsthm, amssymb}
			\usepackage{mathrsfs}
			\usepackage{hyperref} 

			\newtheorem{theorem}{Theorem}
			
			\newtheorem{lemma}{Lemma}
			
			\newtheorem{proposition}{Proposition}

			\theoremstyle{definition}
			
			\newtheorem{example}{Example}
			
			\theoremstyle{remark}
			\newtheorem{remark}{Remark}

			\DeclareMathOperator{\Pn}{Poisson}

			\DeclareMathOperator{\Gam}{Gamma}

			\providecommand{\Li}[2]{{\rm Li}_{#1}\left(#2\right)}

			\DeclareMathOperator{\Var}{Var}

			\providecommand{\Prob}[1]{\mathbb{P}\{#1\}}
			\providecommand{\cum}[1]{\mathscr C\left(#1\right)}
			\providecommand{\cumc}[2]{\mathscr C\left(#1\middle| #2\right)}
			\providecommand{\cumh}[2]{\mathscr C_{#1}\left(#2\right)}
			\providecommand{\cumhc}[3]{\mathscr C_{#1}\left(#2\middle| #3\right)}

			\providecommand{\card}[1]{\texttt{\#}#1}

			\newcommand{\pset}[1]{\mathscr{P}_{#1}}

			\newcommand{\Ex}{\mathbb{E}}

\begin{document}
	\title{On the number of segregating sites}
	\author{Helmut H.~Pitters}
	\address{Institute of Mathematical Stochastics\\
		TU Dresden}
	\email{helmut.pitters@tu-dresden.de}

	
			\begin{abstract}
			  Consider a sample of size $n$ drawn from a large, 
neutral population of haploid individuals subject to mutation at rate $\theta/2$ 
whose genealogy is governed by Kingman's $n$-coalescent. Let $S_n$ count the 
number of segregating sites in this sample under the infinitely many sites model 
of Kimura. For fixed sample size $n$ the main result about $S_n$ is due 
to Watterson~\cite{Watterson1975} who computed its mean and variance as
			  \begin{align*}
			    \Ex S_n &= \theta H_{n-1},\\
			    \Var(S_n) &=\theta H_{n-1}+\theta^2H_{n-1}^{(2)},
			  \end{align*}
			  where $H_n^{(b)}\coloneqq \sum_{k=1}^{n-1}1/k^b$ ($n, b\in\mathbb N$) denotes the generalized harmonic number, and $H_n\coloneqq H_n^{(1)}$ is the (regular) harmonic number. In our main result, Theorem~\ref{thm:sites_cumulants}, we generalize this fact and show that the $i$th cumulant of $S_n$ is given by
			  \begin{align*}
			    \cumh{i}{S_n} &= \sum_{k=1}^{n-1}\Li{1-i}{\frac{\theta}{k+\theta}}=\sum_{b=1}^{i}{i\brace b}(b-1)!\theta^bH_{n-1}^{(b)}\qquad (i\in\mathbb N),
			  \end{align*}
			  where $\Li{n}{u}\coloneqq \sum_{l\geq 1} u^l/l^n$ is the polylogarithm of order $n$, and ${i\brace b}$ denotes the $(i, b)$th Stirling number of the second kind. We find in passing an explicit expression for the cumulants of the negative binomial distribution in terms of the polylogarithm. This seems to be the first explicit formula in the literature for the cumulant of arbitrary order of the negative binomial distribution.
			  
			  As an application of this result we obtain straightforward proofs of the Law of Large Numbers and the Central Limit Theorem for $S_n$.
			\end{abstract}
		\maketitle
			
			\section{Introduction}
			Geneticists often study populations by drawing inferences about the evolution in the past from observations in the present. To be more specific, in neutral populations the genealogy of a sample of $n$ individuals is often approximated by Kingman's $n$-coalescent (and there are rigorous mathematical results justifying this approximation). A verbal description of this coalescent process is as follows. Picture the individuals in the sample labeled $1, \ldots, n$, with a line of descent emanating from each individual and growing at unit speed. At rate one, any pair of individuals merges, i.e.~their lines of descent merge into a single line representing the most recent common ancestor of this pair. After the first merger the process continues with $n-1$ lines of descent following the same dynamics as before. It is clear from this description that the genealogy of a sample of $n$ individuals may be represented as a (random) rooted tree with $n$ leafs labeled $1, \ldots, n.$
			
			In addition to the genealogy mutations are modeled as follows. Conditionally given the genealogical tree (or coalescent tree), run a Poisson process at constant rate $\theta/2>0$, the so-called mutation rate, along the branches of the tree. Each jump of the Poisson process is then interpreted as a mutation affecting any leaf (the individual in the sample) with the property that the unique path connecting the leaf to the root of the tree crosses said mutation.
			
			We restrict ourselves to the infinitely many sites model of Kimura~\cite{Kimura1969}. According to Kimura's model each mutation is thought of as acting on one of infinitely many sites, i.e.~each jump of the Poisson process on the tree introduces a mutation on a site where no mutation was seen before. For detailed expositions of probabilistic models for the evolution of DNA sequences the interested reader is referred to Durrett~\cite{Durrett2008}, Etheridge~\cite{Etheridge2011}, and Tavar\'{e}~\cite{Tavare2004}.

%
%
			
			\section{Some known results}
			Consider Kingman's $n$-coalescent $\Pi_{n}=\{\Pi_n(t), t\geq 0\},$ where $\Pi_n(t)$ is a partition of $[n]\coloneqq \{1, \ldots, n\}$ defined by placing any two individuals into the same block iff their lines of descent have merged up until and including time $t$. If $\tau_k$ denotes the time spent in a state of $k$ blocks by $\Pi_n,$ then $\tau_2, \ldots, \tau_n$ is a sequence of independent exponentials such that $\tau_k$ has parameter $\binom{k}{2}.$
			
			A site is called segregating if it differs in at least two individuals in the sample. We can write the number of segregating sites, $S_{n},$ as
			\begin{align}\label{eq:sites_sum}
			  S_{n} &= \sum_{k=2}^{n} G_{k},
			\end{align}
			where $G_{k}$ counts the number of segregating sites appearing while $\Pi_n$ has $k$ blocks. Drawing on~\cite{Tavare2004} and~\cite{Durrett2008} we briefly recall the basic results known about $S_n$ which go back to Watterson~\cite{Watterson1975}. In terms of the coalescent tree, $G_{k}$ is the number of mutations falling on the $k$ parts of branches of length $\tau_{k}$ each, hence the $G_{2}, \ldots, G_{n}$ are independent. If $\Pi_n$ is in a state of $k$ blocks, the probability to see a mutation before a merger is $\theta/(\theta+k-1)$ since a mutation occurs at rate $k\theta/2$, whereas a merger happens at rate $\binom{k}{2}.$ Consequently, $G_k$ is geometric with success probability $(k-1)/(\theta+k-1)$ (and support $\mathbb N_0$) and mean $\theta/(k-1)$. It is now immediate that $S_n$ has mean $\theta H_{n-1}$ and variance $\theta^2H_{n-1}^{(2)}+\theta H_{n-1}.$ The probability generating function of $S_n$ is given by
			\[\Ex[s^{S_n}] = \prod_{k=1}^{n-1} \frac{k}{k+\theta(1-s)}\qquad (s\in\mathbb R),\]
			and the probability mass function is
			\[\Prob{S_n=m} = \frac{n-1}{\theta}\sum_{k=1}^{n-1} (-1)^{k-1} \binom{n-2}{k-1}\left(\frac{\theta}{k+\theta}\right)^{m+1}\qquad (m\in\mathbb N_0).\]
			
			Alternatively, the conditional distribution of $G_k$ given $\tau_k$ is Poisson with parameter $\theta k\tau_{k}/2$ for which we write
			\begin{align}
			  (G_{k}|\tau_{k}) \sim \Pn(\theta k\tau_{k}/2).
			\end{align}
			This second perspective on the distribution of $G_k$ as a mixture distribution turns out to be fruitful for our study of higher cumulants of $S_n.$
			
			With this geometric view of $S_n,$ and its representation~\eqref{eq:sites_sum} as a sum of indepedent geometric random variables, in analogy to the classical Law of Large Numbers for i.i.d.~random variables it is natural to ask whether $S_n$ can be rescaled (by its mean, say),  in such a way as to converge to some non-degenerate limit. In fact, if we rescale $S_n$ by its mean $\Ex S_n,$ this limit has to be deterministic, hence equal to one, since $\Var(S_n/\Ex S_n)\to 0$ as $n\to\infty.$ To see this, it is enough to show that for any $\epsilon>0$
			\begin{align}
			  \mathbb P\left\{\left|\frac{S_n}{\Ex S_n}-1\right|\geq \epsilon\right\}\leq\epsilon
			\end{align}
			for large enough $n$. Notice first that $H_n\sim\log n,$ and $H_n^{(b)}$ converges to the Riemann zeta function $\zeta(b)\coloneqq\sum_{k=1}^\infty 1/k^b$ as $n\to\infty$ for $b>1$. Here for any two real sequences $(a_n)$ and $(b_n)$ $a_n\sim b_n$ denotes asymptotic equality, i.e.~$\lim_n a_n/b_n=1$. Fix $\epsilon>0$ arbitrarily. Choose an integer $k$ such that $k^2>1/\epsilon$. There exists $N\in\mathbb N$ such that $\sqrt{\Var(S_n/\Ex S_n)}k\leq\epsilon$ for all $n\geq N$, since $\sigma_n^2\coloneqq\Var(S_n/\Ex S_n)=(\theta^2H_{n-1}^{(2)}+\theta H_n)/\theta^2H_n^2\sim 1/\theta\log n$ as $n\to\infty.$
			 Applying Chebyshev's inequality shows
			 \begin{align*}
			   \mathbb P\left\{\left|\frac{S_n}{\Ex S_n}-1\right|\geq \epsilon\right\}\leq \mathbb P\left\{\left|\frac{S_n}{\Ex S_n}-1\right|\geq \sigma_nk\right\}\leq \frac{1}{k^2}\leq\epsilon
			 \end{align*}
			 for all $n\geq N,$ and we have proved the following Law of Large Numbers.
			 \begin{theorem}[Law of Large Numbers]\label{thm:lln}
				As $n\to\infty$ we have convergence
				\begin{align}
				  \frac{S_n}{\theta H_{n-1}} \to 1\text{ almost surely.}
				\end{align}
			\end{theorem}
			\begin{remark}
				The Law of Large Numbers, Theorem~\ref{thm:lln}, suggests that $S_n/H_{n-1}$ could be used as an unbiased estimator for $\theta.$ It is known as \textit{Watterson's estimator}.
			\end{remark}
			Moreover, there is a Central Limit Theorem for $S_n.$
			\begin{theorem}[Central Limit Theorem]\label{thm:clt}
				As $n\to\infty$ we have convergence
				\begin{align}
				  \frac{S_n-\Ex S_n}{\sqrt{\Var(S_n)}}\to_d N,
				\end{align}
				(for all moments and in distribution) where $N$ is a standard Gaussian random variable.
			\end{theorem}
			This result can be proved using the triangular array form of the Central Limit Theorem, cf.~\cite[Theorem 1.23]{Durrett2008}. As an application of our main result, Theorem~\ref{thm:sites_cumulants}, we give a completely different proof via cumulants in Section~\ref{sec:results} for both the Law of Large Numbers as well as the Central Limit Theorem for $S_n$ that boils down to elementary calculations.
			
			\section{Preliminaries}\label{sec:preliminaries}
			Let us introduce some notation. A partition of a set $A$ is a set, $\pi$ say, of non-empty pairwise disjoint subsets of $A$ whose union is $A$. The members of $\pi$ are called the blocks of $\pi.$ Let $\card{A}$ denote the cardinality of $A$ and let $\pset{A}$ denote the set containing all partitions of $A.$

			Before we turn to the results, let us recall the notion of a cumulant. Let $X$ be
			a real random variable whose moment generating function $M_X(s) = \Ex e^{sX}$ exists
			for $|s| < \delta$ and some $\delta > 0$. Recall that the \emph{cumulant generating function} of $X$ is
			defined by $K_X(s)\coloneqq\log M_X (s)$. The coefficient $c_j$ in the power series expansion
			$K_X (s) = \sum_{j=1}^\infty c_j s^j /j!$ near $0$ is called the \emph{$j$th cumulant} of $X$, denoted $\cumh{j}{X}$.
			Notice that $\cumh{1}{X} = \Ex X$ is the expectation and $\cumh{2}{X} = \Var(X)$ the variance of
			$X$.
			
			\begin{example}
				\begin{enumerate}[a)]
					\item Normal distribution. Let $X$ be a mean $\mu,$ variance $\sigma^2>0$ normal random variable. Then $X$ has moment generating function $M_X(s)=\exp(\mu s+\sigma^2s^2/2),$ thus $K_X(s)=\mu s+\sigma^2s^2/2,$ and
					\begin{align*}
					  \cumh{j}{X} &= \begin{cases}
					    \mu & j=1\\
					    \sigma^2 & j=2\\
					    0 & \text{else.}
					  \end{cases}
					\end{align*}
					\item Poisson distribution. Let $N$ be a Poisson random variable with mean $\lambda>0.$ Then $K_N(s)=\lambda(e^s-1),$ therefore, for any $j\in\mathbb N$ $\cumh{j}{N} = \lambda.$
					\item Gamma distribution. Let $T$ be a random variable following a gamma distribution with parameters $\alpha$ and $\beta.$ Then for $s<\beta,$ $M_T(s)=(1-s/\beta)^{-\alpha},$ $K_T(s)=-\alpha\log(1-s/\beta),$ and therefore $\cumh{j}{T} = \alpha(j-1)!/\beta^j.$
				\end{enumerate}
			\end{example}
			We will repeatedly use the following properties of cumulants. For real random
			variables $X, Y$ and $j \in\mathbb N$, $a, b \in\mathbb R$ we have i) $\cumh{j}{X + b} = \cumh{j}{X}$ if $j \geq 2$, ii) $\cumh{j}{aX} = a^j \cumh{j}{X}$, and iii) the independence of $X, Y$ implies $\cumh{j}{X + Y} = \cumh{j}{X} + \cumh{j}{Y}$. More generally, consider a vector $(X_1, \ldots, X_d)$ of real random variables and suppose that the radius of convergence of its moment generating function $M_{(X_1, \ldots, X_d)}(s)\coloneqq \Ex\exp(\sum_{l=1}^d s_lX_l),$ $s\in\mathbb R^d,$ is $\delta>0.$ The cumulant generating function of $(X_1, \ldots, X_d)$ is defined to be $K_{(X_1, \ldots, X_d)}(s)\coloneqq\log M_{(X_1, \ldots, X_d)}(s),$ and the joint cumulant $\cumh{i_1, \ldots, i_d}{X_1, \ldots, X_d}$ of order $(i_1, \ldots, i_d)\in\mathbb N^d$ is the coefficient $c_{i_1, \ldots, i_d}$
			in the series expansion $K_{(X_1, \ldots, X_d)}(s)=\sum_{i_1, \ldots, i_d\geq 1}c_{i_1, \ldots, i_d}\prod_{j=1}^d (s_j^{i_j}/i_j!)$ in $s$ near $0$. Often in the literature $\cumh{1, \ldots, 1}{X_1, \ldots, X_d}$ is called the \emph{joint cumulant} of $(X_1, \ldots, X_d),$ and we also denote it by $\cum{X_1, \ldots, X_d}.$ Notice that in the special case where $X_1=X_2=\cdots=X_d=X$, $\cum{X_1, \ldots, X_d}=\cumh{d}{X}$. There are settings where the joint distribution of $(X_1, \ldots, X_d)$ may be intricate, but there exists another random quantity, $Y$ say, such that the conditional distribution $(X_1, \ldots, X_d|Y)$ is considerably ``simpler'' than the unconditional distribution, e.g.~in the sense that there is a simple generating model for the conditional distribution. A natural question to ask is whether the conditional cumulant $\cumc{X_1, \ldots, X_d}{Y},$ which denotes the cumulant of the conditional distribution of $X_1, \ldots, X_d$ given $Y,$ can be used to compute $\cum{X_1, \ldots, X_d}.$ For $d=2$ and $X_1=X_2=X$ the answer is given by $\cumh{2}{X}=\cumh{1}{\cumhc{2}{X}{Y}}+\cumh{2}{\cumhc{1}{X}{Y}}$, which is nothing but the familiar identity
			\begin{align*}
			  \Var(X) = \Ex\Var(X|Y) + \Var(\Ex[X|Y]),
			\end{align*}
			sometimes called the \emph{law of total variance}. The general identity is due to Brillinger~\cite{Brillinger1969}. Some authors call it the \emph{law of total cumulance}.
			\begin{proposition}[Law of total cumulance; Proposition 4.4 in~\cite{Speed1983}]\label{prop:total_cumulance}
				\begin{align}
					\cum{X_1, \ldots, X_d} &= \sum_{\pi\in\mathscr P_{[d]}} \cum{\cumc{X_b\colon b\in B}{Y}\colon B\in\pi}.
				\end{align}
			\end{proposition}
			In what follows the notation a[i times] is a shorthand for i consecutive symbols
			“a” seperated by commas.
			
			\section{Results}\label{sec:results}
			Before we work out the cumulants of $S_{n}$ let us have a closer look at the $G_{k}.$ Recall the well-known fact that the mixture of a Poisson distribution with gamma mixing distribution yields a negative binomial distribution. More precisely, if $X$ is a gamma random variable with parameters $a, b>0,$ i.e.~$X$ has density $t\mapsto b^{a}t^{a-1}e^{-tb}/\Gamma(a)1\{t>0\},$ and $G$ is a random variable such that
			\begin{align}
			  (G|X=x)\sim \Pn(x),
			\end{align}
			that is, conditionally on $X=x,$ $G$ has a Poisson distribution with parameter $x,$ then $G$ is supported on $\mathbb N_{0}$ and, setting $p\coloneqq b/(1+b)\in (0, 1],$ $G$ has probability mass function
			\begin{align}
			  \Prob{G=k} &= \binom{a+k-1}{k}(1-p)^{a}p^{k}\quad (k\in\mathbb N_{0}).
			\end{align}
			We say that $G$ has a \emph{negative binomial distribution} with parameters $a$ and $p.$ Since $\theta k\tau_{k}/2$ has an exponential distribution with parameter $2\binom{k}{2}/k\theta=(k-1)/\theta$, $G_{k}$ has a negative binomial distribution with parameters $1$ and $(k-1)/(k-1+\theta)$, which is nothing but a geometric distribution with success probability $\theta/(k-1+\theta)$, as we have seen earlier.
			
			Somewhat surprisingly, the author could not find an explicit expression (as opposed to recursive expressions) for the cumulants of a negative binomial distribution in the literature. It turns out that the higher order cumulants of the negative binomial distribution can be expressed in terms of the polylogarithm, as we will show in Proposition~\ref{prop:negative_binomial_cumulants}.
			
			The \emph{polylogarithm of non-integral order}, also known as \emph{Jonqui\`{e}res function}, is defined for complex $s$ and $u\in (-1, 1)$ by
			\begin{align}\label{def:polylogarithm}
			  \Li{s}{u}\coloneqq \sum_{l\geq 1}\frac{u^{l}}{l^{s}}.
			\end{align}
			It is a particular case of \emph{Lerch's function}, also known as \emph{Lerch transcendent},
			\[\Phi(u, s, \alpha) \coloneqq \sum_{l=0}^\infty \frac{u^l}{(l+\alpha)^s}\qquad u\in  (-1, 1), -\alpha\notin\mathbb N_0,\]
			defined for complex $u, s$ and $\alpha.$ In fact, for $\alpha=1$ one obtains $\Li{s}{u}=s\Phi(u, s, 1)$. Notice the particular cases $\Li{0}{u}=u/(1-u),$ $\Li{1}{u}=-\log(1-u),$ and for $\Re s>1,$ $\Li{s}{1}=\sum_{l\geq 1}1/l^s=\zeta(s),$ the \emph{Riemann zeta function}. Lewin~\cite{Lewin1981, Lewin1991} reviews the polylogarithm function. Zagier~\cite{Zagier2007} provides a detailed discussion of the \emph{dilogarithm}, i.e.~$\Li{2}{u}$, and another overview of its properties together with applications in mathematics and physics may be found in~\cite{Maximon2003}. In what follows we will consider the polylogarithm of negative order, i.e.~the special case where $-s\in\mathbb N_0$.
			
			\begin{proposition}[Cumulants of negative binomial distribution]\label{prop:negative_binomial_cumulants}
			Let $N$ be a random variable with negative binomial distribution with parameters $a>0$ and $p\in [0, 1).$ Then the $i$th cumulant of $N$ is given by
			\begin{align}
			  \cumh{i}{N} &= a\Li{1-i}{p}.
			\end{align}
			In particular, $\Ex N=\cumh{1}{N}=ap/(1-p),$ and $\Var(N)=\cumh{2}{N}=ap/(1-p)^2.$
			\end{proposition}
			We write ${n\brace k}$ for the $(n, k)$th \emph{Stirling number of the second kind} counting the number of partitions into $k$ blocks of a set of size $n$. In order to prove Proposition~\ref{prop:negative_binomial_cumulants} we need the following Lemma.
			\begin{lemma}
				For the polylogarithm of negative order we have
				\begin{align}\label{eq:polylog}
				  \Li{-n}{u} &= \sum_{k=0}^n k!{n+1\brace k+1}\left(\frac{u}{1-u}\right)^{k+1}\qquad (n\in\mathbb N_0).
				\end{align}
			\end{lemma}
			\begin{proof}
				We prove the statement by induction on $n$. The polylogarithm of order zero is nothing but the geometric series, hence for $n=0$ we have
				$\Li{0}{u} = \sum_{l\geq 1} u^l = u/(1-u)$
				in agreement with the right hand side in~\eqref{eq:polylog}. To conclude the induction step, suppose that the statement holds for some nonnegative integer $n.$ Taking the derivative of the polylogarithm with respect to $u$ yields the recursion
				\[\frac{d}{du}\Li{-n}{u}=\frac{d}{du}\sum_{l\geq 1} u^ll^n=\sum_{l\geq 1}u^{l-1}l^{n+1}=u^{-1}\sum_{l\geq 1}u^ll^{n+1} = u^{-1}\Li{-(n+1)}{u}.\]
				This recursion together with the initial value $\Li{0}{u}$ completely determines the polylogarithm of negative order. If we can show that the right hand side in~\eqref{eq:polylog},
				\begin{align*}
				  M_{-n}(u) \coloneqq \sum_{k=0}^n k!{n+1\brace k+1}\left(\frac{u}{1-u}\right)^{k+1}\qquad (n\in\mathbb N_0),
				\end{align*}
				satisfies the same recursion, the claim follows. Notice that for $m\in\mathbb N$
				\begin{align*}
				\frac{d}{du}u^m(1-u)^{-m} &= mu^{m-1}(1-u)^{-m}+mu^m(1-u)^{-m-1}\\
				&=u^{-1}\left(m\left(\frac{u}{1-u}\right)^m+m\left(\frac{u}{1-u}\right)^{m+1}\right).
				\end{align*}
				Therefore,
				\begin{align*}
				  \frac{d}{du}M_{-n}(u) &= \frac{d}{du}\sum_{k=0}^n k!{n+1\brace k+1}\left(\frac{u}{1-u}\right)^{k+1}\\
				  &= u^{-1}\sum_{k=0}^n k!{n+1\brace k+1}\left((k+1)\left(\frac{u}{1-u}\right)^{k+1} + (k+1)\left(\frac{u}{1-u}\right)^{k+2}\right)\\
				  &= u^{-1}\sum_{k=0}^n (k+1)!{n+1\brace k+1}\left(\frac{u}{1-u}\right)^{k+1}\left(1+\frac{u}{1-u}\right)\\
				  &= \frac{1}{u(1-u)}\sum_{k=0}^n (k+1)!{n+1\brace k+1}\left(\frac{u}{1-u}\right)^{k+1}\\
				  &= u^{-1}M_{-(n+1)}(u),
				\end{align*}
				where the last equation is seen as follows. Recall for $k>0$ the recurrence relation ${n+1\brace k}={n\brace k-1}+{n\brace k}k$ for the Stirling numbers of the second kind. This implies
				\begin{align*}
				  M_{-(n+1)}(u) &= \sum_{k=0}^{n+1} k!{n+2\brace k+1}\left(\frac{u}{1-u}\right)^{k+1}\\
				  &= \sum_{k=0}^{n+1} k!{n+1\brace k}\left(\frac{u}{1-u}\right)^{k+1}+\sum_{k=0}^{n+1} k!(k+1){n+1\brace k+1}\left(\frac{u}{1-u}\right)^{k+1}\\
				  &= \sum_{k=1}^{n+1} k!{n+1\brace k}\left(\frac{u}{1-u}\right)^{k+1}+\sum_{k=0}^{n} (k+1)!{n+1\brace k+1}\left(\frac{u}{1-u}\right)^{k+1},\\
				\end{align*}
				since ${n+1\brace 0}={n+1\brace n+2}=0.$ Moreover, we have
				\begin{align*}
				  \sum_{k=1}^{n+1} k!{n+1\brace k}\left(\frac{u}{1-u}\right)^{k+1} &= \sum_{k=0}^n (k+1)!{n+1\brace k+1}\left(\frac{u}{1-u}\right)^{k+2}\\
				  &= \frac{u}{1-u}\sum_{k=0}^n (k+1)!{n+1\brace k+1}\left(\frac{u}{1-u}\right)^{k+1},
				\end{align*}
				and consequently,
				\[M_{-(n+1)}(u) = \frac{1}{1-u}\sum_{k=0}^n (k+1)!{n+1\brace k+1}\left(\frac{u}{1-u}\right)^{k+1}.\]
			\end{proof}
			
			\begin{proof}(of Proposition 2)
			Let $\tau$ denote a random variable with $\Gam(\alpha, \beta)$ distribution. Conditionally given $\tau$ let $X$ be a random variable following a $\Pn(\tau)$ distribution.
			  Recall that the $j$th cumulant of a random variate with $\Pn(\alpha)$ law is $\alpha$ for all $j\in\mathbb N$, and the $j$th cumulant of a random variate that obeys a $\Gam(\alpha, \beta)$ law is $\alpha(j-1)!/\beta^j$, $j\in\mathbb N.$ As we have seen earlier, choosing $\beta\coloneqq p/(1-p)$, $X$ and $N$ are equal in distribution. Using this construction and applying the law of total cumulance, Proposition~\ref{prop:total_cumulance}, we find
			  \begin{align*}
			    \cumh{i}{X} &= \cum{X[i\text{ times}]} = \sum_{\pi\in\mathscr P_{[i]}} \cum{\cumhc{\card B}{X}{\tau}\colon B\in\pi}
			    = \sum_{\pi\in\mathscr P_{[i]}} \cum{\tau \colon B\in\pi}\\
			    &= \sum_{\pi\in\mathscr P_{[i]}} \cumh{\card \pi}{\tau} = \alpha\sum_{b=1}^i \frac{(b-1)!}{\beta^b}{i\brace b}=\alpha\Li{1-i}{p}.
			  \end{align*}
			\end{proof}
			
			Knowledge of the formula in Propostion~\ref{prop:negative_binomial_cumulants} for the cumulant of the negative binomial distribution can guide one to find an alternative proof via the cumulant generating function of $N$.
			
			\begin{proof} (second proof of Proposition~\ref{prop:negative_binomial_cumulants})
				The moment generating function of $N$ is given by
				\begin{align*}
				\Ex e^{tN} &= \left(\frac{1-p}{1-pe^t}\right)^\alpha\qquad (t<-\log p).
				\end{align*}
				Thus, the cumulant generating function of $N$ is
				\begin{align*}
				\log \Ex e^{tX} &= \alpha\log\frac{1-p}{1-pe^t} = \alpha(\log (1-p)-\log (1-pe^t)).
				\end{align*}
				Now, for $t<-\log p,$ using the Taylor series of the ordinary logarithm around 1, $\log\frac{1}{1-x}=\sum_{k\geq 1}x^k/k$ for $-1\leq x<1$, one obtains
				\begin{align*}
				-\log(1-pe^t) &= \sum_{k\geq 1}\frac{(pe^t)^k}{k} = \sum_{k\geq 1}\frac{p^k}{k}\sum_{l\geq 0}\frac{(tk)^l}{l!} =\sum_{l\geq 0}\frac{t^l}{l!}\sum_{k\geq 1}p^kk^{l-1}\\
				&= \sum_{k\geq 1}\frac{p^k}{k}+\sum_{l\geq 1}\frac{t^l}{l!}\Li{{1-l}}{p} = -\log(1-p) + \sum_{l\geq 1}\frac{t^l}{l!}\Li{1-l}{p},
				\end{align*}
				and the claim follows.
			\end{proof}
			
			We now turn to the cumulants of $S_{n}.$
			\begin{theorem}[Cumulants of number of segregating sites]\label{thm:sites_cumulants}
			The $i$th cumulant of the total number $S_{n}$ of segregating sites is given by
			\begin{align}
			  \cumh{i}{S_{n}} &= \sum_{k=1}^{n-1} \Li{1-i}{\frac{\theta}{k+\theta}}= \sum_{b=1}^{i}{i\brace b}(b-1)!\theta^{b}H_{n-1}^{(b)},
			\end{align}
			where $H_{n}^{(b)}\coloneqq \sum_{k=1}^{n}1/k^{b}$ denotes the \emph{generalized harmonic number}. In particular, we have
			\begin{align}
			  \Ex S_{n} &= \theta H_{n-1},\\
			  \Var(S_{n}) &= \theta H_{n-1} + \theta^{2}H_{n-1}^{(2)},
			\end{align}
			in agreement with~\cite[equations (1.20) and (1.22)]{Durrett2008}.
			\end{theorem}
			\begin{proof}
			Recall that $G_k$ follows a geometric distribution with success probability $(k-1)/(k-1+\theta)$. Hence, as in the proof of Proposition~\ref{prop:negative_binomial_cumulants} we have
			\begin{align*}
			  \cumh{i}{G_{k}} &= \sum_{b=1}^{i}{i\brace b}(\frac{\theta}{k-1})^{b}(b-1)!=\Li{1-i}{\frac{\theta}{k-1+\theta}},
			\end{align*}
			Since the summands $G_{k}$ in~\eqref{eq:sites_sum} are independent, we have
			\begin{align}
			\cumh{i}{S_{n}}=\sum_{k=2}^{n} \cumh{i}{G_{k}} = \sum_{b=1}^{i}{i\brace b}\theta^{b}(b-1)!\sum_{k=1}^{n-1}1/k^{b}.
			\end{align}
			The claim follows.
			\end{proof}
			We now apply Theorem~\ref{thm:sites_cumulants} to derive the asymptotic behaviour of the number of segregating sites $S_n$ as the sample size grows without bounds.
			
			\section{Applications}\label{sec:applications}
			We first rederive the Law of Large Numbers for $S_n$, Theorem~\ref{thm:lln}.
			
			\begin{proof} (of the Law of Large Numbers)
			For a sequence of random variables almost sure convergence to a constant is equivalent to convergence in distribution.
			
			It is enough to show convergence of all cumulants as long as the limiting distribution is uniquely determined by its cumulants (see e.g.~\cite{Grimmett1992}), which is trivially the case in our statement. For any integer $i\geq 1$ we find
			  \begin{align*}
			    \cumh{i}{\frac{S_n}{\theta H_{n-1}}}=\sum_{b=1}^i {i\brace b}(b-1)!\theta^b\frac{H_{n-1}^{(b)}}{(\theta H_{n-1})^i}\to \begin{cases}
			      1 & \text{if }i=1,\\
			      0 & \text{if }i>1,
			    \end{cases}
			  \end{align*}
			  and the claim follows.
			\end{proof}
			
			Similarly, we obtain a derivation of the Central Limit Theorem for $S_{n},$ Theorem~\ref{thm:clt}.
			
			\begin{proof} (of Theorem~\ref{thm:clt})
			Recall that $H_n\sim\log n$ and $H_{n}^{(b)}\to\zeta(b)$ for $b>1$ as $n\to\infty.$ For $i> 2$ Theorem~\eqref{thm:sites_cumulants} implies
			\begin{align*}
			  \cumh{i}{(S_{n}-\Ex S_n)/\sqrt{\Var(S_n)}} &=\cumh{i}{S_n}/(\sqrt{\Var(S_n)})^i \\&=\sum_{b=1}^{i}\frac{{i\brace b}(b-1)!\theta^{b}H_{n-1}^{(b)}}{(\theta H_{n-1}+\theta^2 H_{n-1}^{(2)})^{\frac{i}{2}}}\\
			  &\sim \theta^{1-i/2}H_{n-1}^{1-i/2}
			  \sim\frac{\theta^{1-i/2}}{(\log n)^{(i-2)/2}}\to 0,
			\end{align*}
			as $n\to\infty.$ From Theorem~\eqref{thm:sites_cumulants} it follows that the mean and variance of ${(S_n-\Ex S_n)/\sqrt{\Var (S_n)}}$ are $0$ and $1$, respectively. Since the $i$th moment of a random variable is a continuous function of its first $i$ cumulants, the claim follows.
			\end{proof}
			
			{\bf Acknowledgements.} The author thanks Martin M\"{o}hle for pointing out the second proof of Proposition~\ref{prop:negative_binomial_cumulants}.
			
			\bibliographystyle{alpha}
			\bibliography{literature}
			
			\end{document}